\documentclass[12pt]{amsart}
\usepackage[utf8]{inputenc}

\usepackage{amssymb,amsthm,amsfonts,enumerate, epsfig,epstopdf,caption, graphicx, xcolor}
\usepackage{amsmath}
\usepackage{subcaption}
\usepackage{hyperref}

\newcommand{\Z}{{\mathbb{Z}}}

\newtheorem{theorem}{Theorem}[section]

\newtheorem{lemma}[theorem]{Lemma}
\newtheorem{conjecture}[theorem]{Conjecture}
\newtheorem{remark}[theorem]{Remark}

\title[Ivanov's finite-index Frattini conjecture]{Ivanov's  finite-index Frattini conjecture for subgroups of mapping class groups }
\author{Renaud Detcherry and Efstratia Kalfagianni}
\address[]{Université Bourgogne Europe, CNRS, IMB UMR 5584, F-21000 Dijon, France \& Institut Universitaire de France}
\email{renaud.detcherry@u-bourgogne.fr}
\address[]{Department of Mathematics, Michigan State University, East
	Lansing, MI, 48824, USA}
\email[]{kalfagia@msu.edu}

\begin{document}
\maketitle

\begin{abstract}
Let $S_g$ be a closed oriented surface of genus $g$. In \cite{Ivanov}, Ivanov conjectured that if $G$ is a finitely generated subgroup of $\mathrm{Mod}(S_g),$ then its finite-index Frattini subgroup $\Phi_f(G)$ is nilpotent. We prove this conjecture, using recent partial results of the first author on the Andersen-Masbaum-Ueno conjecture \cite{Detcherry}.
\end{abstract}

\section{Introduction}
\label{sec:intro}
For $G$ a group, the Frattini subgroup $\Phi(G)$ is the intersection of all maximal subgroups of $G$, while the \emph{finite-index Frattini subgroup} $\Phi_f(G)$ is the intersection of all finite index maximal subgroups of $G$. Frattini subgroups play a central role in the theory of finite and infinite groups. On the other hand, one striking result about finite-index Frattini subgroups is the following classical theorem of Platonov \cite{Platonov}:

\begin{theorem}
	\label{thm:Platonov}\cite{Platonov} For any finitely generated linear group $G$, the subgroup $\Phi_f(G)$ is nilpotent.
\end{theorem} 

In this paper, the infinite groups we are concerned with are mapping class groups of surfaces, and their subgroups.

Let $S=S_g$ be a closed oriented surface of genus $g\ge 3$. The mapping class group $\mathrm{Mod}(S)$ is the group of isotopy classes of orientation-preserving homeomorphisms of $S$. Let $J_1(S)$ denote the Torelli subgroup of $\mathrm{Mod}(S)$, which consists of all elements that act trivially on $H_1(S).$ Finally, let $J_2(S)$ be the Johnson kernel, which is the subgroup generated by separating Dehn twists.

It is well known that the groups $\mathrm{Mod}(S)$ share many properties with finitely generated linear groups (among which residual finiteness, Tits alternative and virtual torsion-freeness).

In \cite[Thm 10.7]{IvanovSurvey}, Ivanov proved that the ordinary Frattini subgroup, $\Phi(G)$, is nilpotent for any finitely generated subgroup $G$ of $\mathrm{Mod}(S_g)$. Ivanov then asked whether the stronger statement, using $\Phi_f(G)$ instead of $\Phi(G),$ is also true \cite[Problem 15]{Ivanov}. The latter is the analogue of Platonov's theorem for subgroups of $\mathrm{Mod}(S).$ In this paper, we prove Ivanov's conjecture:

\begin{theorem} \label{thm:Ivanov} Let $g\geq 1$ and $G$ be a finitely generated subgroup of $\mathrm{Mod}(S_g)$. Then $\Phi_f(G)$ is nilpotent.
\end{theorem}

\begin{remark}
	The above theorem implies the analogous statement replacing the closed surface $S_g$ by a compact oriented surface $S_{g,b}$ with genus $g$ and $b\geq 1$ boundary components, since $\mathrm{Mod}(S_{g,b})$ embeds in  $\mathrm{Mod}(S_{g+b})$.
\end{remark}
Note that $\mathrm{Mod}(S_1)\simeq \mathrm{SL}_2(\Z)$ is linear. By \cite{BigelowBudney} or \cite{Korkmaz}, $\mathrm{Mod}(S_2)$ is linear. Thus Ivanov's conjecture for $g\leq 2$ is a consequence of Platonov's theorem. 

In \cite{MasbaumReid}, Masbaum and Reid settled Ivanov's conjecture for normal subgroups of $\mathrm{Mod}(S_g)$ and subgroups containing a finite index subgroup of the Torelli subgroup $J_1(S_g)$. The method used by Masbaum and Reid involves the use of the so-called \emph{$\mathrm{SO}(3)$-Witten-Reshetikhin-Turaev quantum representations} at a primitive $p$-th root of unity
$$\rho_p:\mathrm{Mod}(S_g)\longrightarrow \mathrm{PU}(d_p)$$
where $p\geq 5$ is a prime number and $d_p$ is an explicit integer. Those are the associated projective mapping class group representations of the $\mathrm{SO}_3$-WRT TQFTs; we refer to \cite{BHMV}  for a definition. By a result of Larsen and Wang \cite{LarsenWang}, those representations have dense images for $g\geq 2$. In \cite{MasbaumReid}, the authors deduce that for $g\geq 3$, the group $\mathrm{Mod}(S_g)$, and some of its subgroups, are residually finite simple, which implies that they have trivial finite-index Frattini subgroups.

In contrast to Masbaum and Reid's strategy, our proof does not use results about Zariski density for quantum representations. Instead, we rely on partial results on the Andersen-Masbaum-Ueno conjecture (AMU conjecture), which were recently obtained by the first author in \cite{Detcherry}. The AMU conjecture may be stated as follows\footnote{The version we give of the AMU conjecture is a slightly weakened one, since the original statement also considers the representations $\rho_r$ for all odd integers $r$.}:

\begin{conjecture}\label{conj:AMU} \cite{AMU} Let $f\in \mathrm{Mod}(S)$ be a mapping class with a pseudo-Anosov part. Then $\rho_p(f)$ has infinite order for all large enough primes $p$.
\end{conjecture}

In \cite{Detcherry}, the following partial result was obtained:

\begin{theorem}
	\label{thm:AMU-Johnson}\cite[Theorem 1.1]{Detcherry} Let $f$ be a non-trivial element of $[J_2(S),J_2(S)].$ Then $\rho_p(f)$ has infinite order for all large enough primes $p$.
\end{theorem}

The short proof of Theorem \ref{thm:Ivanov}, which will be given in Section \ref{sec:IvanovProof}, rests on Theorem \ref{thm:AMU-Johnson}, a Lie theoretic lemma (Lemma \ref{lemma:NilpotentUnitary}), and two lemmas about finite-index Frattini subgroups (Lemma \ref{lemma:FrattiniEpimorphism} and Lemma \ref{lemma:FrattiniVirtualAb}).

\textbf{Acknowledgements:} The IMB, host institution of the first author, receives support from the EIPHI Graduate School (contract ANR-17-EURE-0002). The first author thanks Gregor Masbaum, Alan Reid and Ramanujan Santharoubane for helpful discussions. The research of the second author is partially supported by the
NSF grants DMS-2304033,  DMS-2603192  and by a Simons Foundation travel grant.

\section{A Lie-theoretic lemma}
\label{sec:compactLieNilpotent}
One striking property of the representations $\rho_p$ is that they are unitary representations but have infinite images. In our proof of Theorem \ref{thm:Ivanov}, we will take advantage of unitarity through the following lemma:

\begin{lemma}\label{lemma:NilpotentUnitary} Let $N$ be a nilpotent subgroup of a compact Lie group. Then $[N,N]$ is finite.
\end{lemma}
The above lemma should be completely standard to Lie groups experts. However, we were not able to find a reference for this exact statement, so we provide its elementary proof.

\begin{proof}
	Let $\overline{N}$ be the closure of $N$. If $N$ is nilpotent of nilpotency class $c$, then, since iterated commutator maps are continuous, we have that $\overline{N}$ is also nilpotent of class at most $c$. 
	
	Let $\mathfrak{n}$ be the Lie algebra of $\overline{N},$ and $N_0$ its identity component. Let 
	$$\mathrm{Ad}: \overline{N} \longrightarrow \mathrm{Aut}(\mathfrak{n})$$
	be the map induced by conjugation. Averaging a positive definite symmetric bilinear form on $\mathfrak{n}$ along a Haar measure, one gets an $\mathrm{Ad}(\overline{N})$-invariant positive definite symmetric bilinear form on $\mathfrak{n},$ see \cite[Proposition 4.24]{Knapp} for details. Thus $\mathrm{Ad}(x)$ is an orthogonal map for all $x\in \overline{N}.$ We claim that $\mathrm{Ad}(x)$ is unipotent as well. Indeed, for $x\in \overline{N},$ consider the commutator map
	$$\begin{array}{rccl}
		F_x: & N_0 & \longrightarrow & N_0
		\\ & y & \longrightarrow & xyx^{-1}y^{-1}
	\end{array}$$
	If $\overline{N}$ is nilpotent of class $c,$ then $(F_x)^{\circ c}$ is the constant map equal to $1.$ However, $F_x(1)=1$ and $dF_x(1)=\mathrm{Ad}(x)-\mathrm{id}_{\mathfrak{n}},$ thus 
	$$(d(F_x)^{\circ c})(1)=(\mathrm{Ad}(x)-\mathrm{id}_{\mathfrak{n}})^c=0.$$
	Now, the linear map $\mathrm{Ad}(x)$ is orthogonal and unipotent, hence $\mathrm{Ad}(x)=\mathrm{id}_{\mathfrak{n}}$ for all $x\in \overline{N},$ and thus $N_0$ is central in $\overline{N}$.
	
	Now, since $\overline{N}$ is compact, $\overline{N}/N_0$ is finite. By Schur's theorem (\cite{Schur}, see also \cite[Theorem 5.3]{Neumann}), if $G$ is a group such that $G/Z(G)$ is finite then $[G,G]$ is finite. Thus $[\overline{N},\overline{N}]$ is finite.

	
	

	Finally, $[N,N]\subset [\overline{N},\overline{N}],$ so $[N,N]$ is finite as well.
\end{proof}

\section{Two lemmas about $\Phi_f$}
\label{sec:groupTheoryLemma}

We give two important lemmas about finite-index Frattini subgroups, the first of which concerns their behavior with respect to morphisms:  

\begin{lemma}
	\label{lemma:FrattiniEpimorphism}\cite[Lemma 2.2]{MasbaumReid}: Let $\varphi: G\longrightarrow H$ be a group morphism. Then $\varphi(\Phi_f(G))\subset \Phi_f(\varphi(G)).$
\end{lemma} 

Our second lemma concerns the interaction between finite-index Frattini subgroups and virtually abelian normal subgroups. 

\begin{lemma}
	\label{lemma:FrattiniVirtualAb} Let $H$ be a group, and $V\triangleleft H$ be a virtually abelian finitely generated subgroup of $\Phi_f(H)$. Then $V$ is nilpotent.
\end{lemma}

\begin{proof} Since $V$ is virtually abelian and finitely generated, let $A_0\simeq \Z^r$ be a free abelian subgroup of finite  index in $V.$ Since $V$ is finitely generated, the subgroup $A_0$ contains a finite-index subgroup $A$ which is characteristic in $V,$ and therefore normal in $H$.

We first prove that $V/mA$ is nilpotent for any $m\geq 1.$ Indeed, consider the morphism

$$\alpha: H\longrightarrow \mathrm{Aut(V/mA)}$$
induced by conjugation. Then Lemma \ref{lemma:FrattiniEpimorphism} implies that 
$\alpha(\Phi_f(H))\subset \Phi_f(\alpha(H)).$ However, $\alpha(H)\subset \mathrm{Aut}(V/mA)$ is finite. By \cite[Thm 9.3(a)]{DH}, we get that $\alpha(\Phi_f(H))$ is nilpotent, thus so is 
$$\alpha(V)\simeq (V/mA)/Z(V/mA).$$ Thus $V/mA$ is nilpotent.

In particular, for $m=1$ we get that $V/A$ is nilpotent. Now, let us prove that $A\subset Z(V),$ which would imply that $V$ is nilpotent.

Let $d=[V:A],$ then for any prime $l\nmid d,$ the group $A/lA$ is an abelian $l$-Sylow subgroup of the nilpotent group $V/lA,$ hence it is central. Therefore, $[V,A]\subset lA$ for all $l\nmid d,$ which implies that $A\subset Z(V)$ and concludes the proof.
	
\end{proof}
\section{Proof of Ivanov's finite-index Frattini conjecture}
\label{sec:IvanovProof}

\begin{proof}[Proof of Theorem \ref{thm:Ivanov}]
	Let $S$ be a closed surface of genus $g\geq 3$ and let $G$ be a finitely generated subgroup of $\mathrm{Mod}(S).$ We set 
	
	$$N:=\Phi_f(G), \ B:=N\cap J_1(S), \ C:=N\cap J_2(S)$$
	 We first argue that $C$ is abelian. 
	 
	For any prime $p\geq 5$, let $\rho_p$ be the $\mathrm{SO}(3)$-quantum representation at a primitive $p$-th root of unity. By Lemma \ref{lemma:FrattiniEpimorphism}, we have $\rho_p(N)=\rho_p(\Phi_f(G)) \subset \Phi_f(\rho_p(G))$. However, by Platonov's theorem, $\Phi_f(\rho_p(G))$ is nilpotent, hence so is $\rho_p(N)$.
	
	Now, $\rho_p(N)$ is a subgroup of a projective unitary group, which is compact. Thus, by Lemma \ref{lemma:NilpotentUnitary}, we have that $[\rho_p(C),\rho_p(C)]=\rho_p([C,C])$ is finite. However, $[C,C] \subset [J_2(S),J_2(S)]$, thus by Theorem \ref{thm:AMU-Johnson},  any non-trivial element of $[C,C]$ has an image of infinite order under $\rho_p$ for some prime $p\geq 5$. It follows that $[C,C]=1$, that is, $C$ is abelian. 
	
	Now, $B/C\simeq BJ_2(S)/J_2(S)$ is a subgroup of $J_1(S)/J_2(S),$ which is abelian by \cite{Johnson} (see also \cite[Theorem A]{Putman}).
	
	Finally, consider $q:G\longrightarrow \mathrm{Sp}_{2g}(\Z)$ the restriction of the homology representation. Then, $q(N)\simeq N/B,$ but by Lemma \ref{lemma:FrattiniEpimorphism}, we have
	$$ q(N)=q(\Phi_f(G))\subset \Phi_f(q(G))$$
	with the latter group nilpotent by Platonov's theorem. Therefore $N/B$ is nilpotent.
	
	Since $C$ and $B/C$ are abelian, and $N/B$ is nilpotent, we conclude that $N$ is solvable. By \cite[Theorems A and B]{BLM83}, this implies that $N$ is virtually abelian and finitely generated. 
	
	Finally, $N=\Phi_f(G)$ is virtually abelian, finitely generated and $N\triangleleft G$, thus we can apply Lemma \ref{lemma:FrattiniVirtualAb} to conclude that $N$ is nilpotent.\end{proof}

\begin{remark}
	Let $\overline{\mathrm{Mod}}(S)$ be the group of non-necessarily orientation-preserving isotopy classes of homeomorphisms of $S.$ Then for any $G\subset \overline{\mathrm{Mod}}(S)$ finitely generated, one has $\Phi_f(G)$ nilpotent as well. The proof above adapts to this case, with the following changes:
	If $G$ is not included in $\mathrm{Mod}(S),$ note that $G\cap \mathrm{Mod}(S)$ has index $2$, thus is a maximal subgroup. Hence $N=\Phi_f(G)\subset \mathrm{Mod}(S).$ Moreover, the homology representation extends to $q:\overline{\mathrm{Mod}}(S)\longrightarrow \mathrm{GL}_{2g}(\Z)$, and one can replace the use of the quantum representation $\rho_p$ with $\mathrm{Ind}_{\mathrm{Mod}(S)}^{\overline{\mathrm{Mod}}(S)}\mathrm{ad}(\rho_p).$
\end{remark}

\bibliographystyle{hamsalpha}

\bibliography{biblio}

@book {IvanovSurvey,
    AUTHOR = {Ivanov, Nikolai V.},
     TITLE = {Subgroups of {T}eichm\"uller modular groups},
    SERIES = {Translations of Mathematical Monographs},
    VOLUME = {115},
      NOTE = {Translated from the Russian by E. J. F. Primrose and revised
              by the author},
 PUBLISHER = {American Mathematical Society, Providence, RI},
      YEAR = {1992},
     PAGES = {xii+127},
      ISBN = {0-8218-4594-2},
   MRCLASS = {57M50 (20F38 30F60 57N05)},
  MRNUMBER = {1195787},
MRREVIEWER = {Athanase\ Papadopoulos},
       DOI = {10.1090/mmono/115},
       URL = {https://doi-org.proxy-bu2.ube.fr/10.1090/mmono/115},
}

@article {Korkmaz,
    AUTHOR = {Korkmaz, Mustafa},
     TITLE = {On the linearity of certain mapping class groups},
   JOURNAL = {Turkish J. Math.},
  FJOURNAL = {Turkish Journal of Mathematics},
    VOLUME = {24},
      YEAR = {2000},
    NUMBER = {4},
     PAGES = {367--371},
      ISSN = {1300-0098,1303-6149},
   MRCLASS = {57M60 (57N05)},
  MRNUMBER = {1803819},
MRREVIEWER = {Bruno\ P.\ Zimmermann},
}

@article {BLM83,
    AUTHOR = {Birman, Joan S. and Lubotzky, Alex and McCarthy, John},
     TITLE = {Abelian and solvable subgroups of the mapping class groups},
   JOURNAL = {Duke Math. J.},
  FJOURNAL = {Duke Mathematical Journal},
    VOLUME = {50},
      YEAR = {1983},
    NUMBER = {4},
     PAGES = {1107--1120},
      ISSN = {0012-7094,1547-7398},
   MRCLASS = {20F38 (30F20 57N05)},
  MRNUMBER = {726319},
MRREVIEWER = {William\ Harvey},
       DOI = {10.1215/S0012-7094-83-05046-9},
       URL = {https://doi-org.proxy-bu2.ube.fr/10.1215/S0012-7094-83-05046-9},
}

@book {DH,
    AUTHOR = {Doerk, Klaus and Hawkes, Trevor},
     TITLE = {Finite soluble groups},
    SERIES = {De Gruyter Expositions in Mathematics},
    VOLUME = {4},
 PUBLISHER = {Walter de Gruyter \& Co., Berlin},
      YEAR = {1992},
     PAGES = {xiv+891},
      ISBN = {3-11-012892-6},
   MRCLASS = {20D10 (20F17)},
  MRNUMBER = {1169099},
MRREVIEWER = {R.\ Bryce},
       DOI = {10.1515/9783110870138},
       URL = {https://doi-org.proxy-bu2.ube.fr/10.1515/9783110870138},
}

@article {Putman,
    AUTHOR = {Putman, Andrew},
     TITLE = {The {J}ohnson homomorphism and its kernel},
   JOURNAL = {J. Reine Angew. Math.},
  FJOURNAL = {Journal f\"ur die Reine und Angewandte Mathematik. [Crelle's
              Journal]},
    VOLUME = {735},
      YEAR = {2018},
     PAGES = {109--141},
      ISSN = {0075-4102,1435-5345},
   MRCLASS = {57N05 (20J06 57M07)},
  MRNUMBER = {3757472},
MRREVIEWER = {S\'ergio\ Tadao\ Martins},
       DOI = {10.1515/crelle-2015-0017},
       URL = {https://doi-org.proxy-bu2.ube.fr/10.1515/crelle-2015-0017},
}

@article {BHMV,
    AUTHOR = {Blanchet, C. and Habegger, N. and Masbaum, G. and Vogel, P.},
     TITLE = {Topological quantum field theories derived from the {K}auffman
              bracket},
   JOURNAL = {Topology},
  FJOURNAL = {Topology. An International Journal of Mathematics},
    VOLUME = {34},
      YEAR = {1995},
    NUMBER = {4},
     PAGES = {883--927},
      ISSN = {0040-9383},
   MRCLASS = {57N10 (57M25 81T40)},
  MRNUMBER = {1362791},
MRREVIEWER = {Justin\ D.\ Roberts},
       DOI = {10.1016/0040-9383(94)00051-4},
       URL = {https://doi-org.proxy-bu2.ube.fr/10.1016/0040-9383(94)00051-4},
}

@incollection {Ivanov,
    AUTHOR = {Ivanov, Nikolai V.},
     TITLE = {Fifteen problems about the mapping class groups},
 BOOKTITLE = {Problems on mapping class groups and related topics},
    SERIES = {Proc. Sympos. Pure Math.},
    VOLUME = {74},
     PAGES = {71--80},
 PUBLISHER = {Amer. Math. Soc., Providence, RI},
      YEAR = {2006},
      ISBN = {978-0-8218-3838-9; 0-8218-3838-5},
   MRCLASS = {57M07 (20F65 57M05)},
  MRNUMBER = {2264532},
MRREVIEWER = {Eric\ Michael\ Katerman},
       DOI = {10.1090/pspum/074/2264532},
       URL = {https://doi-org.proxy-bu2.ube.fr/10.1090/pspum/074/2264532},
}

@article {MasbaumReid,
    AUTHOR = {Masbaum, G. and Reid, A. W.},
     TITLE = {Frattini and related subgroups of mapping class groups},
   JOURNAL = {Tr. Mat. Inst. Steklova},
  FJOURNAL = {Trudy Matematicheskogo Instituta Imeni V. A. Steklova},
    VOLUME = {292},
      YEAR = {2016},
     PAGES = {149--158},
      ISSN = {0371-9685,3034-1809},
      ISBN = {5-7846-0137-7; 978-5-7846-0137-7},
   MRCLASS = {20F34 (20E15)},
  MRNUMBER = {3628458},
MRREVIEWER = {A.\ S.\ Kondrat\cprime ev},
       DOI = {10.1134/S037196851601009X},
       URL = {https://doi-org.proxy-bu2.ube.fr/10.1134/S037196851601009X},
}

@article {BigelowBudney,
    AUTHOR = {Bigelow, Stephen J. and Budney, Ryan D.},
     TITLE = {The mapping class group of a genus two surface is linear},
   JOURNAL = {Algebr. Geom. Topol.},
  FJOURNAL = {Algebraic \& Geometric Topology},
    VOLUME = {1},
      YEAR = {2001},
     PAGES = {699--708},
      ISSN = {1472-2747,1472-2739},
   MRCLASS = {20F36 (20C15 57M07)},
  MRNUMBER = {1875613},
MRREVIEWER = {Luis\ Paris},
       DOI = {10.2140/agt.2001.1.699},
       URL = {https://doi-org.proxy-bu2.ube.fr/10.2140/agt.2001.1.699},
}

@article {Johnson,
    AUTHOR = {Johnson, Dennis},
     TITLE = {The structure of the {T}orelli group. {II}. {A}
              characterization of the group generated by twists on bounding
              curves},
   JOURNAL = {Topology},
  FJOURNAL = {Topology. An International Journal of Mathematics},
    VOLUME = {24},
      YEAR = {1985},
    NUMBER = {2},
     PAGES = {113--126},
      ISSN = {0040-9383},
   MRCLASS = {57N05 (57M99)},
  MRNUMBER = {793178},
MRREVIEWER = {J.\ S.\ Birman},
       DOI = {10.1016/0040-9383(85)90049-7},
       URL = {https://doi-org.proxy-bu2.ube.fr/10.1016/0040-9383(85)90049-7},
}

@article {LarsenWang,
    AUTHOR = {Larsen, Michael and Wang, Zhenghan},
     TITLE = {Density of the {SO}(3) {TQFT} representation of mapping class
              groups},
   JOURNAL = {Comm. Math. Phys.},
  FJOURNAL = {Communications in Mathematical Physics},
    VOLUME = {260},
      YEAR = {2005},
    NUMBER = {3},
     PAGES = {641--658},
      ISSN = {0010-3616,1432-0916},
   MRCLASS = {57R56 (22E47 57M99)},
  MRNUMBER = {2183960},
MRREVIEWER = {R\u azvan\ Gelca},
       DOI = {10.1007/s00220-005-1382-x},
       URL = {https://doi-org.proxy-bu2.ube.fr/10.1007/s00220-005-1382-x},
}

@book {Knapp,
    AUTHOR = {Knapp, Anthony W.},
     TITLE = {Lie groups beyond an introduction},
    SERIES = {Progress in Mathematics},
    VOLUME = {140},
   EDITION = {Second},
 PUBLISHER = {Birkh\"auser Boston, Inc., Boston, MA},
      YEAR = {2002},
     PAGES = {xviii+812},
      ISBN = {0-8176-4259-5},
   MRCLASS = {22-01},
  MRNUMBER = {1920389},
}

@misc{Detcherry,
      title={The {A}ndersen-{M}asbaum-{U}eno conjecture for the derived subgroup of the {J}ohnson kernel}, 
      author={Renaud Detcherry},
      year={2026},
      eprint={arxiv:2603.29397},
      archivePrefix={arXiv},
      primaryClass={math.GT},
      url={https://arxiv.org/abs/2603.29397}, 
}

@article {Neumann,
    AUTHOR = {Neumann, B. H.},
     TITLE = {Groups with finite classes of conjugate elements},
   JOURNAL = {Proc. London Math. Soc. (3)},
  FJOURNAL = {Proceedings of the London Mathematical Society. Third Series},
    VOLUME = {1},
      YEAR = {1951},
     PAGES = {178--187},
      ISSN = {0024-6115,1460-244X},
   MRCLASS = {20.0X},
  MRNUMBER = {43779},
MRREVIEWER = {R.\ M.\ Thrall},
       DOI = {10.1112/plms/s3-1.1.178},
       URL = {https://doi-org.proxy-bu2.ube.fr/10.1112/plms/s3-1.1.178},
}

@article {Schur,
    AUTHOR = {Schur, J.},
     TITLE = {\"Uber die {D}arstellung der endlichen {G}ruppen durch
              gebrochen lineare {S}ubstitutionen},
   JOURNAL = {J. Reine Angew. Math.},
  FJOURNAL = {Journal f\"ur die Reine und Angewandte Mathematik. [Crelle's
              Journal]},
    VOLUME = {127},
      YEAR = {1904},
     PAGES = {20--50},
      ISSN = {0075-4102,1435-5345},
   MRCLASS = {99-04},
  MRNUMBER = {1580631},
       DOI = {10.1515/crll.1904.127.20},
       URL = {https://doi-org.proxy-bu2.ube.fr/10.1515/crll.1904.127.20},
}

@article {Platonov,
    AUTHOR = {Platonov, V. P.},
     TITLE = {Frattini subgroup of linear groups and finitary
              approximability},
   JOURNAL = {Dokl. Akad. Nauk SSSR},
  FJOURNAL = {Doklady Akademii Nauk SSSR},
    VOLUME = {171},
      YEAR = {1966},
     PAGES = {798--801},
      ISSN = {0002-3264},
   MRCLASS = {20.75},
  MRNUMBER = {212096},
MRREVIEWER = {J.\ Mennicke},
}

@article {AMU,
    AUTHOR = {Andersen, J\o rgen Ellegaard and Masbaum, Gregor and Ueno,
              Kenji},
     TITLE = {Topological quantum field theory and the {N}ielsen-{T}hurston
              classification of {$M(0,4)$}},
   JOURNAL = {Math. Proc. Cambridge Philos. Soc.},
  FJOURNAL = {Mathematical Proceedings of the Cambridge Philosophical
              Society},
    VOLUME = {141},
      YEAR = {2006},
    NUMBER = {3},
     PAGES = {477--488},
      ISSN = {0305-0041,1469-8064},
   MRCLASS = {57R56 (17B37 57M27)},
  MRNUMBER = {2281410},
MRREVIEWER = {Yasuyuki\ Kawahigashi},
       DOI = {10.1017/S0305004106009698},
       URL = {https://doi-org.proxy-bu2.ube.fr/10.1017/S0305004106009698},
}

\end{document}